\documentclass[12pt, a4paper]{amsart}
\usepackage[utf8]{inputenc}
\usepackage{amsthm}
\usepackage{mathrsfs}
\usepackage[none]{hyphenat}
\usepackage[english]{babel}
\usepackage{setspace}
\usepackage{pinlabel}

\usepackage[a4paper, top=3cm, bottom=4cm,left=3cm,right=3cm]{geometry}

\usepackage{comment}

\usepackage{pgf,tikz}
\usepackage{microtype}

\usepackage{tikz-cd}
\usepackage{tikz,pgfplots}
\usetikzlibrary{positioning, angles}
\usetikzlibrary{arrows}
\usetikzlibrary{decorations.pathmorphing, decorations.pathreplacing}
\usetikzlibrary{decorations.markings}
\usetikzlibrary{calc,shapes}
\usepackage{mathtools}
\usepackage{lipsum}
\usepackage{comment}
\usepackage{float}    
\usetikzlibrary{shapes.geometric}
\usepackage{tikz}
\usepackage{graphicx}
\let\phi\varphi

\usepackage{enumerate}
\usepackage{caption}
\usepackage{subcaption}

\usepackage{amsmath, amsthm, amssymb, amsfonts, xfrac}
\newtheorem{prop}{\textsc{Proposition}}[section]
\newtheorem{cor}[prop]{\textsc{Corollary}}
\newtheorem{thm}[prop]{\textsc{Theorem}}

\newtheorem{defn}[prop]{\textsc{Definition}}

\newtheorem{remark}[prop]{\textsc{Remark}}

\newtheorem*{thm*}{\textsc{Theorem}}

\newcommand{\Z}{\mathbb{Z}}
\newcommand{\Q}{\mathbb{Q}}

\newcommand{\R}{\mathbb{R}}

\newcommand{\C}{\widetilde{\mathcal{C}}}

\DeclareMathOperator{\lk}{lk}

\DeclareMathOperator{\sign}{sign}

\DeclareMathOperator{\Diffeo}{Diffeo}

\DeclareMathOperator{\Fix}{Fix}

\let\epsilon\varepsilon

\title[A new invariant of equivariant concordance]{A new invariant of equivariant concordance and results on 2-bridge knots}
\author[]{}
\begin{document}
\maketitle
\begin{abstract}
We study the equivariant concordance classes of two-bridge knots, providing an easy formula to compute their butterfly polynomial \cite{boyle2021equivariant}, and we give two different proofs that no two-bridge knot is equivariantly slice. 
Finally, we introduce a new invariant of equivariant concordance for strongly invertible knots. Using this invariant as an obstruction we strengthen the result on two-bridge knots, proving that their equivariant concordance order is always infinite. 

\end{abstract}

\section{Introduction}
A strongly invertible knot is a pair $(K,\rho)$ where $K\subseteq S^3$ is a knot and $\rho\in \Diffeo^+(S^3)$ is an  involution such that $\rho(K)=K$ and $\rho$ reverses the orientation on $K$. By the solution of the Smith conjecture \cite{smith} it is known that $\Fix(\rho)$ is an unknot which intersects $K$ in two points. In \cite{sakuma} the author gave a well defined notion of \emph{equivariant connected sum} for strongly invertible knots by endowing them with a \emph{direction}.
Furthermore, Sakuma \cite{sakuma} studied strongly invertible knots up to \emph{equivariant concordance} and introduced the \emph{equivariant concordance group} $\C$.

The equivariant concordance group is far from being understood. However, Di Prisa proved in \cite{diprisa} that $\C$ is not abelian and several authors defined new invariants for equivariant concordance and obstructions for equivariant sliceness, see for example  \cite{boyle2021equivariant,alfieri2021strongly,dai_hedden_mallick,dai_mallick_stoffregen,miller_powell}. In particular, Boyle and Issa \cite{boyle2021equivariant} defined the \emph{butterfly link} associated with a directed strongly invertible knot, and used it to define several equivariant concordance invariants.

In this paper we study some of this invariants in the case of $2$-bridge knots. In Proposition \ref{prop:formula} we provide a formula to compute the \emph{butterfly polynomial} (see \cite{boyle2021equivariant}) for $2$-bridge knots. Our initial goal was to prove Proposition \ref{no_2_slice} combining the obstructions given by Sakuma's \emph{$\eta$-function} (see \cite{sakuma}) and by the butterfly polynomial. This approach was inconclusive. However, using this formula we prove Corollary \ref{surjectivity}, which is the analogous of \cite[Theorem II]{sakuma} for the butterfly polynomial.

The main result of the paper is Theorem \ref{infinite_order}, which states the following:
\begin{thm*}
Let $K$ be a directed strongly invertible knot and let $\widehat{L_b}(K)$ be its butterfly link endowed with the opposite semi-orientation. If the Conway polynomial of $\widehat{L_b}(K)$ is non-zero then $K$ is not equivariantly slice and has infinite order in $\C$.
\end{thm*}

Using this result we are able to prove Proposition \ref{prop:2bri}, showing that every $2$-bridge knot has infinite order in the equivariant concordance group.

\subsection*{Organization of the paper}
Section~\ref{sec:preliminari} contains a brief recap on the some results on directed strongly invertible knots that we need in the following. For the details see \cite{sakuma,boyle2021equivariant}. In Section~\ref{sec:eta} we provide a formula for the computation of the butterfly polynomial (\cite[Definition 4.5]{boyle2021equivariant}) of $2$-bridge knots. In Section~\ref{sec:2bri} we prove in two different ways that every $2$-bridge knot is not equivariantly slice. The first one relies on the \emph{axis-linking number} introduced in \cite{boyle2021equivariant}, while the second one uses the nullity of the butterfly link. Finally, in Section~\ref{sec:moth} we define a new invariant of equivariant concordance for strongly invertible knots. We use this invariant to show that the equivariant concordance order of every $2$-bridge knot is infinite.

\section{Preliminaries}\label{sec:preliminari}

\subsection{Directed strongly invertible knots}
We briefly recall the notion of direction for a strongly invertible knot and of the equivariant concordance group.
\begin{defn}
A \emph{direction} on a strongly invertible knot $(K,\rho)$ is the choice of an oriented half-axis $h$, i.e. one of the two connected components of $\Fix(\rho)\setminus K$.
\end{defn}
We call the triple $(K,\rho,h)$ a \emph{directed strongly invertible knot}. We write $K$ instead of $(K,\rho,h)$ when it is not strictly necessary to specify the choice of strong inversion and direction.
\begin{defn}
Let $(K,\rho,h)$ be a directed strongly invertible knot. We define
\begin{itemize}
    \item the \emph{mirror} of $(K,\rho,h)$ by $mK=(mK,\rho,h)$,
    \item the \emph{axis-inverse} of $(K,\rho,h)$ by $iK=
    (K,\rho,-h)$, where $-h$ is the direction given by the half-axis $h$ with the opposite orientation,
    \item the \emph{antipode} of $(K,\rho,h)$ by $aK=(K,\rho,h')$, where $h'$ is the direction given by the half-axis complementary to $h$. The orientation on $h'$ is the one coherent with $h$.
\end{itemize}
\end{defn}
\begin{defn}
We say that two DSI knots $(K_i,\rho_i,h_i)$, $i=0,1$ are \emph{equivariantly concordant} if there exists a smooth properly embedded annulus $C\cong S^1\times I\subset S^3\times I$, invariant with respect to some involution $\rho$ of $S^3\times I$ such that:
\begin{itemize}
    \item $\partial (S^3\times I,C)=(S^3,K_0)\sqcup -(S^3,K_1)$,
    \item $\rho$ is in an extension of the strong inversion $\rho_0\sqcup\rho_1$ on $S^3\times 0\sqcup S^3\times 1$,
    \item the orientations of $h_0$ and $-h_1$ induce the same orientation on the annulus $\Fix(\rho)$, and $h_0$ and $h_1$ are contained in the same component of $\Fix(\rho)\setminus C$.
\end{itemize}
\end{defn}

The \emph{equivariant concordance group} is the set $\C$ of classes of directed strongly invertible knots up to equivariant concordance, endowed with the operation of \emph{equivariant connected sum}, which we denote by $\widetilde{\#}$ (see \cite{sakuma,boyle2021equivariant} for details).

\begin{remark}
\emph{
It is easy to see that the mirror, axis-reverse and antipode induce involutive maps from the equivariant concordance group to itself. From the definition of equivariant connected sum we can easily deduce the following properties. Given two directed strongly invertible knot $K$ and $J$ we have:
\begin{itemize}
    \item $m(K\widetilde{\#}J)=mK\widetilde{\#}mJ$,
    \item $i(K\widetilde{\#}J)=iJ\widetilde{\#}iK$,
    \item $a(K\widetilde{\#}J)=aJ\widetilde{\#}aK$.
\end{itemize}
Equivalently, we can say that $m$ is an automorphism of $\C$, while $i$ and $a$ are anti-automorphisms.}
\end{remark}

\begin{remark}\label{rem_sliceness}
\emph{
As a consequence, the equivariant concordance order of a directed strongly invertible knot $(K,\rho,h)$ does not depend on the choice of a direction and it does not change by taking the mirror of the knot.}
\end{remark}

\subsection{Butterfly links}
In \cite[Definition 4.1]{boyle2021equivariant} Boyle and Issa associate a directed strongly invertible knot $(K,\rho,h)$ with a $2$-components $2$-periodic link (i.e. the involution $\rho$ exchanges its components), called the \emph{butterfly link} $L_b(K)$, defined as follows.
Take an equivariant band $B$, parallel to the preferred half-axis $h$, which attaches to $K$ at the two fixed points.
Performing a band move on $K$ along $B$ produces a $2$-component link with linking number between components depending on the number of twists of $B$. The $L_b(K)$ is the one obtained from such a band move on $K$, so that the linking number between its components is $0$.
Observe that $\partial B\setminus K$ consists of two arcs parallel to $h$, which we orient as $h$. The arcs lie in different components of $L_b(K)$ and the orientation on each component of $L_b(K)$ is the one induced from the respective arc.

The following result can be proven easily by adapting the proof of \cite[Proposition 7]{boyle2021equivariant}. We report the proof because it will be useful for Proposition \ref{prop:conc_inv}.
\begin{prop}\label{b_link_concordance}
Let $(K_i,\rho_i,h_i)$, $i=0,1$, be two equivariantly concordant directed strongly invertible knots. Then, $L_b(K_0)$ and $L_b(K_1)$ are also equivariantly concordant (as $2$-periodic links).
\end{prop}

\begin{proof}
Let $C\subset S^3\times I$ be a concordance between $(K_0,\rho_0,h_0)$ and $(K_1,\rho_1,h_1)$ invariant with respect to an extension $\rho:S^3\times I\longrightarrow S^3\times I$ of $\rho_0\sqcup\rho_1$.
Let $A=\Fix(\rho)$ be the annulus of fixed points of $\rho$. Observe that $C\cap A=\alpha\cup\beta$, where $\alpha, \beta$ are two curves joining respectively the initial and final points of the half-axes of $K_0,K_1$.

Now $A\setminus(\alpha\cup\beta)$ has two connected components: call $D$ the component containing $h_0$ and $h_1$.
Choose an equivariant tubular neighborhood $N$ of $D$ and observe that $N\cap C$ is an $D^1$-subbundle of $N_{|\alpha\cup\beta}$.
Consider two equivariant bands $B_i\subset S^3\times\{i\}$, $i=0,1$, with $B_i$ intersecting $K_i$ and containing the half-axis $h_i$.
We can choose $B_i$ in such way that $B_i\setminus K_i\subset N$, hence $B_0\cup B_1\cup C$ intersect $N$ in a $D^1$-subbundle of $N_{|\partial D}$.

Choose $B_0$ so that the band move of $K_0$ along $B_0$ produces $L_b(K_0)$, and take $B_1$ so that the $D^1$-subbundle extends over $D$ to an $D^1$-subbundle $E$ of $N$. Call $L$ the $2$-link obtained from $K_1$ by the band move along $B_1$.

Now $E\cong D^1\times D\cong D^1\times D^1\times D^1$, where $0\times \partial D^1\times D^1=\alpha\cup\beta$. Then $C_b=(C\setminus D^1\times\partial D^1\times D^1)\cup \partial D^1\times D^1\times D^1$ is an equivariant concordance between $L_b(K_0)$ and $L$.
Since the linking number between components is a concordance invariant of $2$-component links, we have $L=L_b(K_1)$.
\end{proof}

\begin{cor}\label{iff_equiv_slice}
A directed strongly invertible knot is equivariantly slice if and only if its butterfly link is equivariantly slice (as a $2$-periodic $2$-link).
\end{cor}

\subsection{Strong inversions on two-bridge knots}
Let $K=K(p,q)\subseteq S^3$ be a $2$-bridge knot. From \cite{siebenmann1975exercices} we know that we can write $p/q$ as a continued fraction:
\[ [a_1,\dots,a_n] = a_1 + \cfrac{1}{a_2 + \cfrac{1}{\ddots + \cfrac{1}{a_n} }} \]
where $a_1,\dots,a_n$ and $n$ are even non-zero integers.

Recall that every $2$-bridge knot is simple (see \cite{hatcher1985incompressible}). 
In \cite{sakuma} the author shows that a hyperbolic $2$-bridge knot $K(p,q)$ admits exactly two inequivalent structures of strongly invertible knot, namely the ones described in the diagrams
\[ I_1(a_1,a_3,\dots,a_{n-1};a_2/2,a_4/2,\dots,a_n/2), \]
and
\[ I_1(-a_n,-a_{n-2},\dots,-a_2;-a_{n-1}/2,\dots,-a_3/2,-a_1/2), \]
where the strong involution on a diagram $I_1(\alpha_1,\dots,\alpha_n;c_1\dots,c_n)$ is given by the $\pi$-rotation around the vertical axis (see Figure~\ref{I1a}).

If $K(p,q)$ is a torus knot, it admits one only strong inversion, namely the one described by $I_1(a_1,a_3,\dots,a_{n-1};a_2/2,a_4/2,\dots,a_n/2).$
 \begin{figure}[htt]
        \centering
         \begin{tikzpicture}
          \draw (0,0) node[above right]{\includegraphics[width=0.49\linewidth]{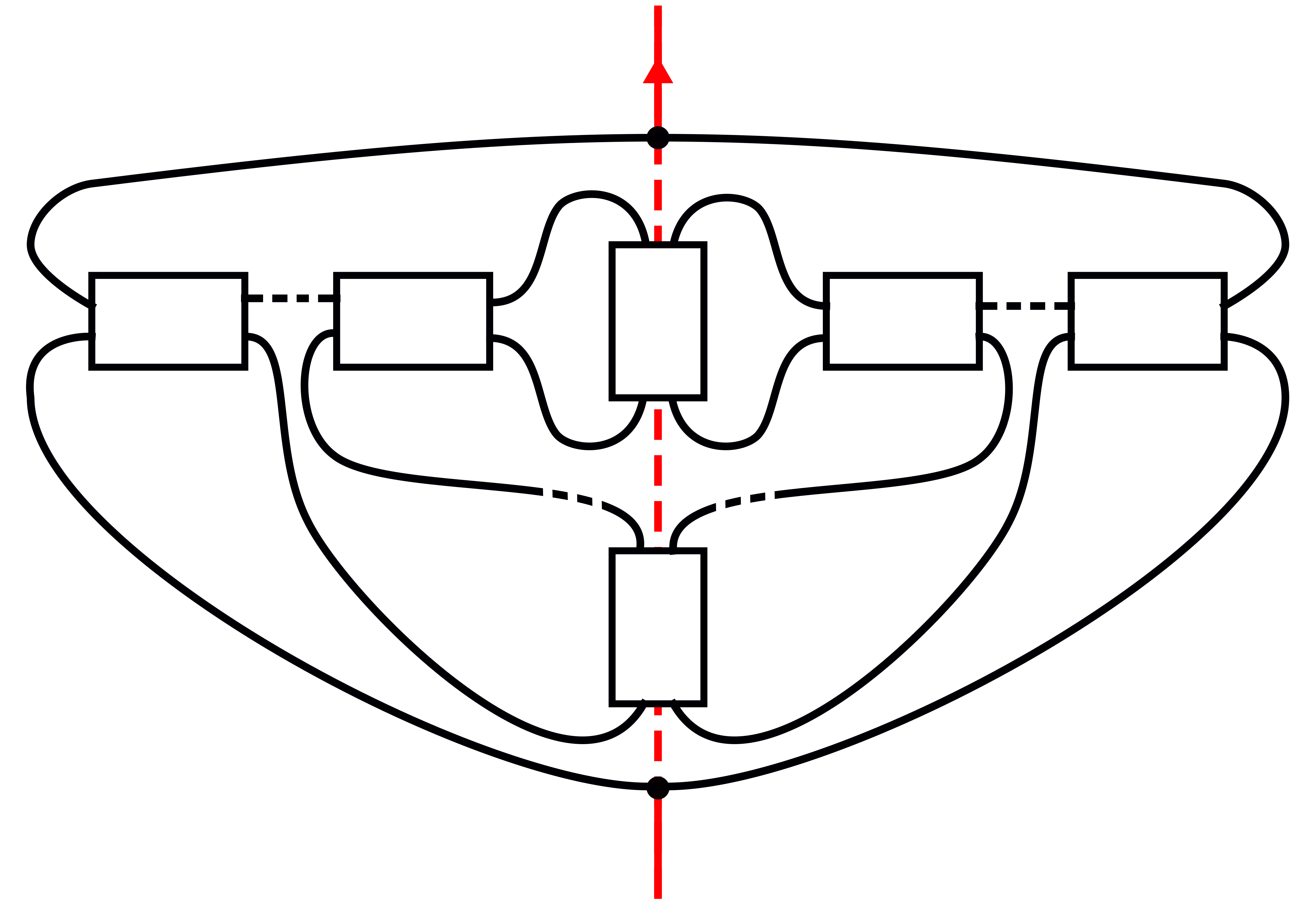}};
   		\draw (1.07,3.52) node (a) {$-c_n$};
   		\draw (2.49,3.52) node (a) {$-c_1$};
   		\draw (6.58,3.52) node (a) {$-c_n$};
   		\draw (5.2,3.52) node (a) {$-c_1$};
            \draw (3.85,3.48) node (a) {$\alpha_1$};
            \draw (3.84,1.77) node (a) {$\alpha_n$};
           \end{tikzpicture}
           \caption{$I_1(\alpha_1,\dots,\alpha_n;c_1\dots,c_n)$.}
           \label{I1a}
        \end{figure}
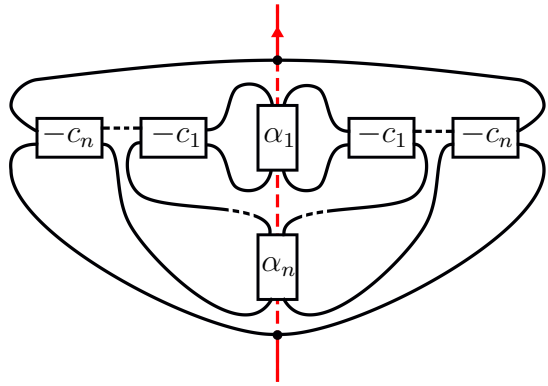
In the following we will consider $I_1(\alpha_1,\dots,\alpha_n;c_1\dots,c_n)$ as a diagram for the directed strongly invertible knot $K=K(p,q)$, with the direction given by the oriented unbounded half-axis in Figure \ref{I1a}, unless the direction is otherwise specified.
Here $n>0$, $\alpha_1,\dots,\alpha_n\in 2\Z\setminus\{0\}$ and $c_1,\dots,c_n\in\Z\setminus\{0\}$.

\section{Eta-function}\label{sec:eta}
Let $K\cup J$ be a $2$-component link with linking number $0$ between components. 
In \cite{kojima_yamasaki} Kojima and Yamasaki introduced the
$\eta$-function, which is a topological concordance invariant for such links.

We briefly recall the construction of this invariant. Let $X_K$ be the complement of $K$ in $S^3$ and let $\widetilde{X}_K
\longrightarrow X_K$ be its infinite cyclic covering. Denote by $t$ a generator of the deck transformation of $\widetilde{X}_K$. Recall that the Alexander module of $K$ is $H_1(\widetilde{X}_K,\Z)$ endowed with the $\Z[t,t^{-1}]$-module structure given by the action of $t$.
Now let $l$ be the canonical longitude of $J$ and let $\widetilde{l}$ and $\widetilde{J}$ be two nearby lifts of $l$ and $J$ to $\widetilde{X}_K$. 

Since $\lk(K,J)=0$, $\widetilde{l}$ and $\widetilde{J}$ are closed curves, hence they can be seen as classes in $H_1(\widetilde{X}_K,\Z)$. Since the Alexander module is a torsion $\Z[t,t^{-1}]$-module, there exists a non-zero $f(t)\in\Z[t,t^{-1}]$ such that $f(t)\cdot\widetilde{l}=0$, i.e. we can find a $2$-chain $\Delta$ such that $\partial \Delta=f(t)\cdot\widetilde{l}$.
Then the $\eta$-function is defined as

$$\eta(K,J;t)=\frac{1}{f(t)}\sum_{n\in\Z}\#(\Delta\cap t^n\widetilde{J})\cdot t^n,$$
where $\#(\Delta\cap t^n\widetilde{J})$ is the algebraic intersection.

One can check that $\eta$ is well defined and that it has the following properties (see \cite{kojima_yamasaki}):

\begin{enumerate}[i)]
    \item $\eta(K,J;t)=\eta(K,J;t^{-1})$,
    \item $\eta(K,J;1)=0$,
    \item $\eta$ does not depend on the orientation of the link,
    \item $\eta$ is an invariant of topological concordance of links.
\end{enumerate}

Observe that in general $\eta$ does depend on the order of the components of the link, i.e. $\eta(K,J;t)\neq\eta(J,K;t)$.

In the following we will denote by $\mathbb{Z}\langle t\rangle\subseteq \mathbb{Z}[t,t^{-1}]$ the subgroup of Laurent polynomials satisfying properties i) and ii).

Boyle and Issa \cite{boyle2021equivariant} defined the \emph{butterfly polynomial} $\eta(L_b(K))$ of a directed strongly invertible knot $K$ as $\eta$-function of its butterfly link (since $L_b(K)$ is $2$-periodic it does not depend on the order of its components) and showed that it induces a group homomorphism
$$\eta(L_b(-)):\C\longrightarrow\Q(t).$$

\begin{remark}
Since the butterfly polynomial induces a homomorphism, every directed strongly invertible knot with non-trivial butterfly polynomial has infinite order in $\C$.
\end{remark}

We describe now a formula for the butterfly polynomial of $2$-bridge knots, similar to the one for Sakuma's $\eta$-function \cite[Proposition 2.3]{sakuma}.
To do so, we report a convenient algorithm \cite{sakuma,snape} to compute $\eta(K,J;t)$ in the case of a $2$-component link $L=K\cup J$ with the component $K$ unknotted.

In this special case the infinite cyclic cover of $X_K$ is diffeomorphic to $\R\times D^2$, and the $\eta$-function of $L$ is simply
\[ \eta(K,J;t) = \sum_{i\in\Z} lk(\widetilde{l},t^i(\widetilde{J}))t^i. \]

The algorithm consists of the following $4$ steps.
\begin{enumerate}
    \item Start by noting that $lk(\widetilde{l},t^i(\widetilde{J}))t^i=lk(\widetilde{J},t^i(\widetilde{J}))t^i$ for $i\neq0$, since $\widetilde{l}$ is a nearby perturbation of $\widetilde{J}$. Therefore, letting $r=\lk(\widetilde{l},J)$, we get
\[ \sum_{i\in\Z}lk(\widetilde{l},t^i(\widetilde{J}))t^i =
   \sum_{i\in\Z\setminus0}lk(\widetilde{J},t^i(\widetilde{J}))t^i + r,
   \]
In the following steps we compute $\overline{\eta}(t)=\sum_{i\in\Z\setminus0}lk(\widetilde{J},t^i(\widetilde{J}))t^i$.
Since $\eta(1)=0$ it is easy to retrieve $r=-\overline{\eta}(1)$.

\item Draw a fundamental domain of the infinite cyclic cover $\widetilde{X}_K$ (see Figure~\ref{fig:fondom}).
\begin{figure}[ht]
    \centering
    \includegraphics[width=0.7
    \linewidth]{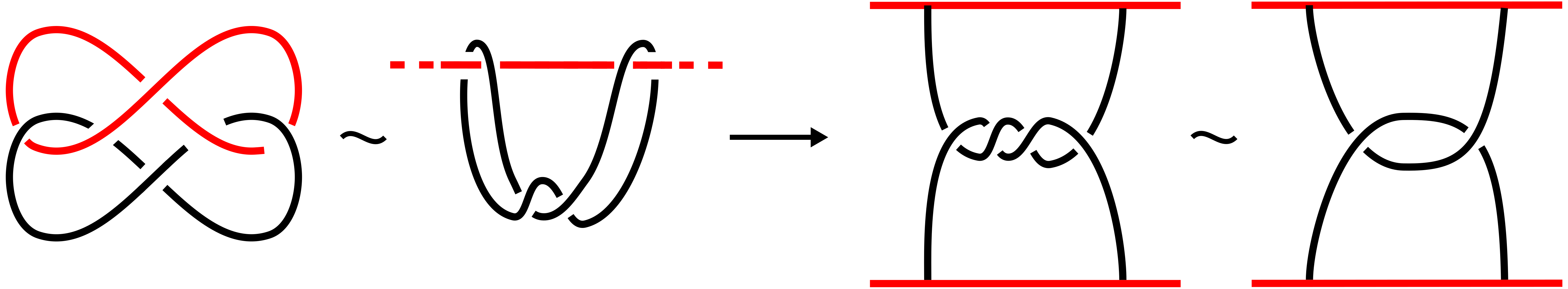}
    \caption{Fundamental domain of the Whitehead link.}
    \label{fig:fondom}
\end{figure}
Assign a label and an orientation to each arc as follows:
\begin{enumerate}[i.]
 \item The arc starting from the top right point has index $0$ and is oriented downwards.
 \item Suppose an arc $\alpha$ is already labelled and oriented. Let $A$ be the end point of $\alpha$ and $B$ be the point opposite to $A$. Call $\beta$ the strand that starts from $B$ (by saying this we are orienting it). Define $\text{index}(\beta)$ (the label on $\beta$) to be $\text{index}(\alpha)+1$ if $B$ lies on the lower side of the domain or $\text{index}(\alpha) - 1$ if $B$ is on the upper side.
\end{enumerate}

The labels we put on the strands keep track of which translate of $\widetilde{J}$ in $\widetilde{X}_K$ they correspond to. A strand labelled by $i$ is a portion of $t^i(\widetilde{J})$. Hence a crossing where an arc labelled by $i$ overcrosses an arc labelled by $j$ corresponds to a crossing between $t^i(\widetilde{J})$ and $t^j(\widetilde{J})$ or, equivalently, between $\widetilde{J}$ and $t^{i-j}(\widetilde{J})$. This motivates the following step.
\item Assign to each double point $P$ a sign $\epsilon_P\in\{+,-\}$ and an integer $d_P\in\mathbb{Z}$ as follows. The sign $\epsilon_P$ is the sign of the crossing and $d_P$ is the difference between the label on the overcrossing arc and the label of the undercrossing arc.
\item Now let \[ \overline{\eta}(t) = \sum_P \epsilon_P t^{d_P} \quad\text{and}\quad r = -\overline{\eta}(1).\] Then $\eta(t)$ is obtained as $\eta(t) = \overline{\eta}(t) + r$.
\end{enumerate}

We will use this algorithm to prove Proposition~\ref{prop:formula}.
     \begin{figure}[htt]
   	 \centering
        \begin{subfigure}{0.49\textwidth}
        \centering
         \begin{tikzpicture}
          \draw (0,0) node[above right]{\includegraphics[width=\linewidth]{diag_I1.png}};
   		\draw (1.07,3.52) node (a) {$-c_n$};
   		\draw (2.49,3.52) node (a) {$-c_1$};
   		\draw (6.58,3.52) node (a) {$-c_n$};
   		\draw (5.2,3.52) node (a) {$-c_1$};
            \draw (3.85,3.48) node (a) {$\alpha_1$};
            \draw (3.84,1.77) node (a) {$\alpha_n$};
           \end{tikzpicture}
           \caption{}
           \label{I1}
        \end{subfigure}
        \begin{subfigure}{0.49\textwidth}
        \centering
         \begin{tikzpicture}
          \draw (0,0) node[above right]{\includegraphics[width=\linewidth]{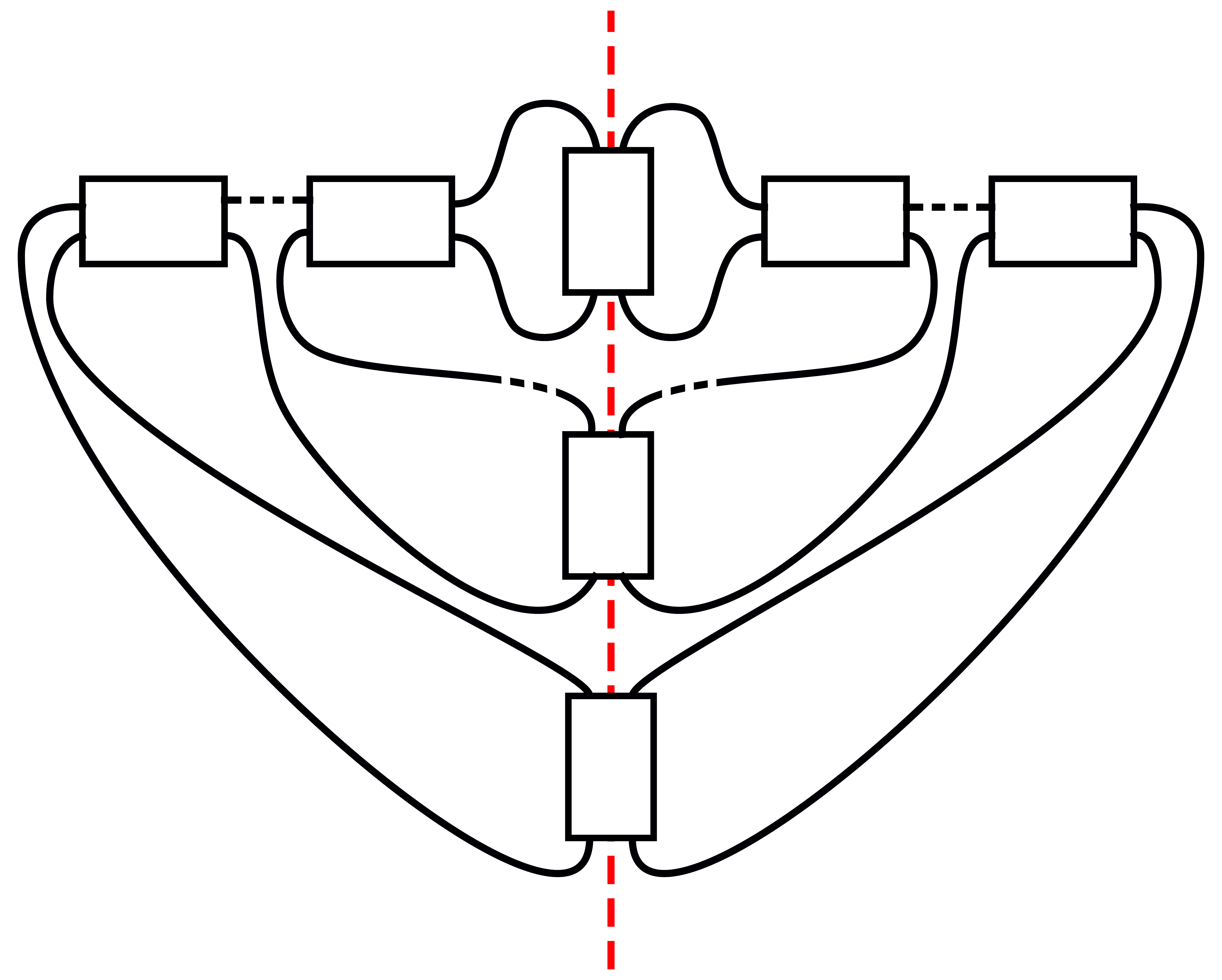}};
   		\draw (1.07,4.72) node (a) {$-c_n$};
   		\draw (2.49,4.72) node (a) {$-c_1$};
   		\draw (6.58,4.72) node (a) {$-c_n$};
   		\draw (5.2,4.72) node (a) {$-c_1$};
            \draw (3.83,4.67) node (a) {$\alpha_1$};
            \draw (3.81,2.95) node (a) {$\alpha_n$};
            \draw (3.81,1.45) node (a) {-$b$};
            \draw (6,0.45) node (a) {$b = \sum_{i=1}^n\alpha_i$};
           \end{tikzpicture}
           \caption{}
           \label{butterfly}
        \end{subfigure}
   	\caption{Strong inversion on a $2$-bridge knot and construction of the butterfly link.}
   	\label{dom}
   	\end{figure} 
\begin{prop}\label{prop:formula}
    Let $I_1(\alpha_1,\dots,\alpha_n;c_1\dots,c_n)$ be a diagram for the directed strongly invertible knot $K=K(p,q)$.
    Then the butterfly polynomial of $K$ is given by:
    \[ \eta_{L_b(K)}(t) = \sum_{i=1}^n c_i \left( t^{\sigma_i} + t^{-\sigma_i}\right) -2\sum_{i=1}^n c_i, \]
    where $\sigma_i = 1/2\sum_{j=1}^i \alpha_j$. 
    \begin{proof}

        Figure~\ref{dom} shows the construction of the butterfly link $L_b(K)= K_0 \sqcup K_1$ starting from $I_1(\alpha_1,\dots,\alpha_n;c_1\dots,c_n)$.
        We only add a box of $-2\sigma_n$ crossings, so that $lk(K_0,K_1)=0$.  
        Applying a sequence of flypes we see that the diagrams in Figure~\ref{label} also represent the link $L_b(K)$. 

     \begin{figure}[ht]
        \centering
         \begin{tikzpicture}
          \draw (0,0) node[above right]{\includegraphics[width=0.65\linewidth]{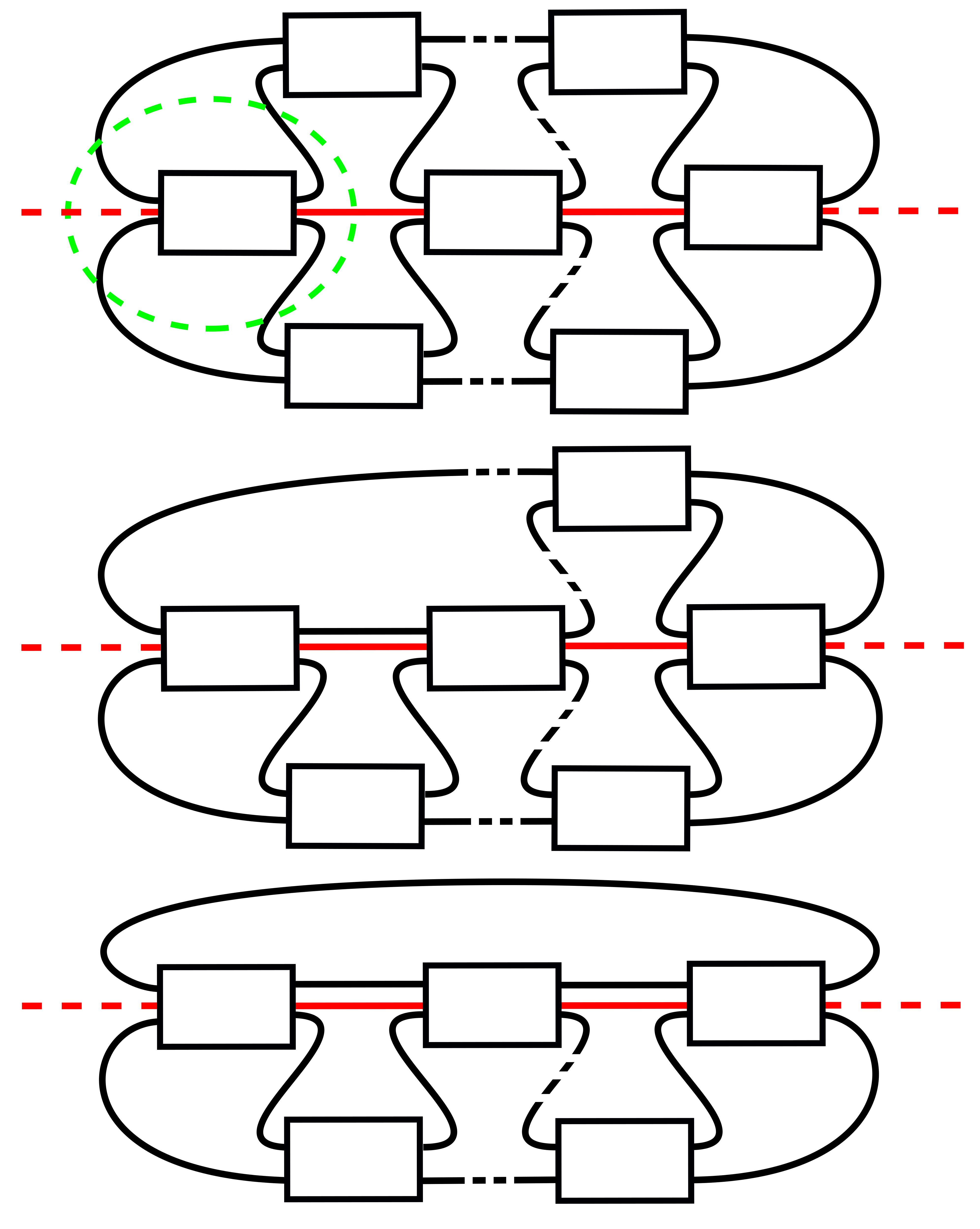}};

        \draw (3.65,11.81) node (a) [scale=1.2] {$-c_1$};
   		\draw (6.38,11.81) node (a) [scale=1.2] {$-c_n$};
   		\draw (3.65,8.69) node (a) [scale=1.2] {$-c_1$};
   		\draw (6.38,8.67) node (a) [scale=1.2] {$-c_n$};
        \draw (2.43,10.21) node (a) [scale=1.2] {$\alpha_1$};
        \draw (5.1,10.21) node (a) [scale=1.2] {$\alpha_2$};
        \draw (7.55,10.26) node (a) [scale=1.2] {$-b$};
      
        \draw (6.38,7.52) node (a) [scale=1.2] {$-c_n$};
   		\draw (3.65,4.35) node (a) [scale=1.2] {-$2c_1$};
   		\draw (6.38,4.35) node (a) [scale=1.2] {$-c_n$};
        \draw (2.43,5.9) node (a) [scale=1.2] {$\alpha_1$};
        \draw (5.1,5.9) node (a) [scale=1.2] {$\alpha_2$};
        \draw (7.55,5.95) node (a) [scale=1.2] {$-b$};
      
        \draw (2.43,2.37) node (a) [scale=1.2] {$\alpha_1$};
   		\draw (5.1,2.37) node (a) [scale=1.2] {$\alpha_2$};
   		\draw (3.65,0.79) node (a) [scale=1.2] {-$2c_1$};
   		\draw (6.42,0.79) node (a) [scale=1.2] {-$2c_n$};
   		\draw (7.55,2.37) node (a) [scale=1.2] {$-b$};
        \draw (9.2,0.45) node (a) {$b = \sum_{i=1}^n\alpha_i$};
           
           \end{tikzpicture}
           \caption{Flypes on the butterfly link.}
        \label{label}
     \end{figure}
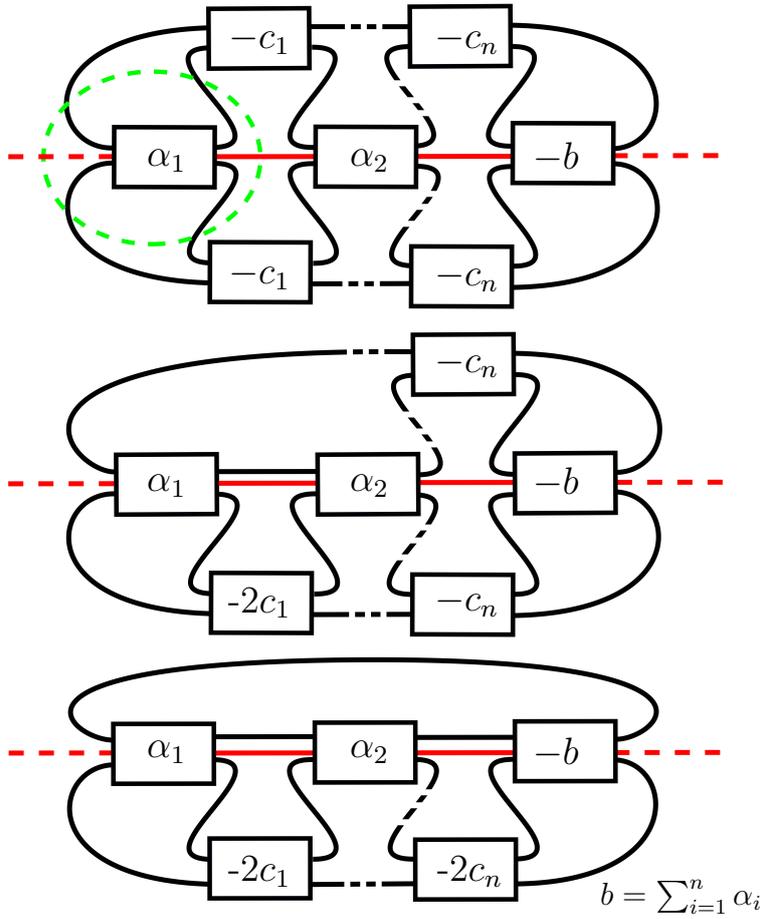

Since the link $L_b(K)$ has unknotted components (see Figure \ref{label}), we can compute the $\eta$-function drawing a fundamental domain (Figure~\ref{fundom}) and applying the algorithm previously described.
        
        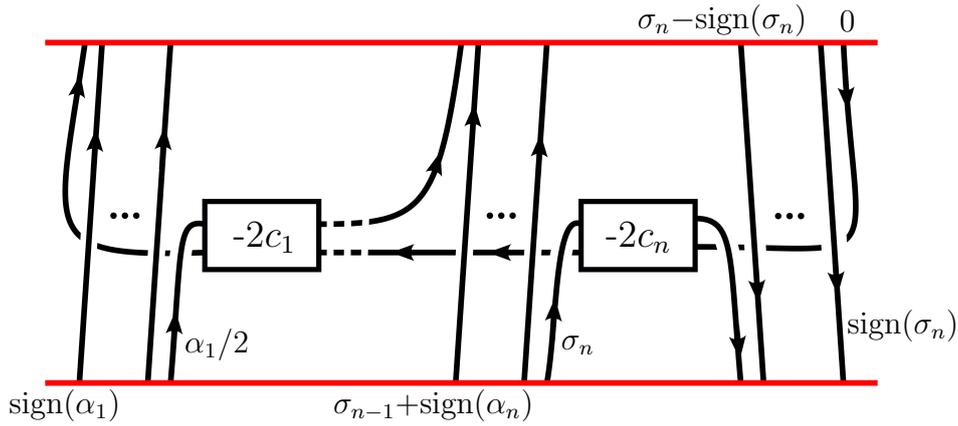
\begin{figure}[ht]
        \centering
         \begin{tikzpicture}
          \draw (0,0) node[above right]{\includegraphics[width=0.75\linewidth]{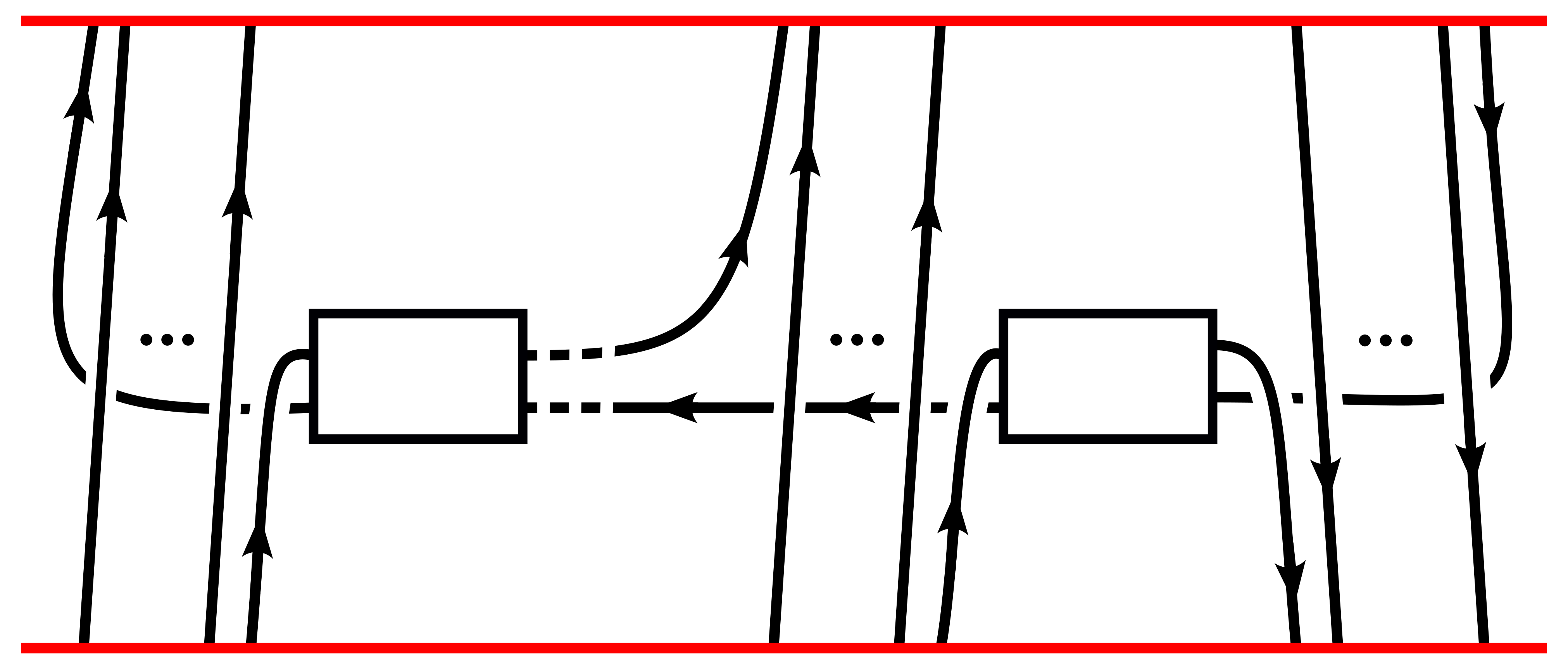}};
   		\draw (3.18,2.22) node (a) [scale=1.2] {-$2c_1$};
   		\draw (8.1,2.22) node (a) [scale=1.2] {-$2c_n$};
   		\draw (10.85,5.1) node (a) {$0$};
            \draw (0.55,0) node (a) {sign$(\alpha_1)$};
            \draw (2.55,0.8) node (a) {$\alpha_1/2$};      
            \draw (5.4,0) node (a) {$\sigma_{n-1}$+sign($\alpha_n$)};
            \draw (7.3,0.8) node (a) {$\sigma_{n}$};
            \draw (11.58,1.05) node (a) {sign($\sigma_{n}$)};
            \draw (9.2,5.1) node (a) {$\sigma_{n}-$sign($\sigma_{n}$)};
           \end{tikzpicture}
           \caption{Fundamental domain with labelled oriented arcs.}
           \label{fundom}
        \end{figure}

        We start by labelling and orienting the arcs.
        Call $\gamma$ the top-right arc, labelled zero. The arc $\gamma$ runs across each $-2c_i$ box reaching the left side of the domain. If $\alpha_1>0$, $\gamma$ rises and ends in the upper horizontal bar. Then we will run across $\alpha_1/2$ arcs riding from the lower bar to the upper bar, labelled with increasing indexes. If $\alpha_1<0$, $\gamma$ descends ending in the lower horizontal bar and the following $\alpha_1/2$ arcs ride from the upper bar to the lower one with decreasing labels. This means that in both cases the last arc of this group is labelled by $\alpha_1/2$. 

        This goes for all the groups of $\alpha_i/2$ vertical arcs: in each group the labels increase or decrease by one and the arc entering the $-2c_i$ box, after the group of $\alpha_i/2$ vertical arcs, is labelled by $\sum_{j=1}^i \alpha_j/2 = \sigma_i$.

        The arc exiting the $n$-th box is labelled by $\sigma_n = \sum_{i=1}^n \alpha_i/2$. The last group of vertical arcs consists of exactly $|2\sigma_n|$ arcs oriented upwards or downwards depending on whether $\sigma_n<0$ or $\sigma_n>0$. It follows that the labels increase, or decrease, by one until they reach $0$, at this point we meet the first arc, which is already oriented and labelled.

        At last, we count the crossings.
        We must count the crossings in the groups of vertical arcs and in the boxes for $i\in\{1,\dots,n\}$ and in the last group of $|2\sigma_n|$ vertical arcs.
        For each $i\in\{1,\dots,n\}$ in the $i$-th box we find that the two strands run in opposite direction, one labelled $0$ and one labelled $\sigma_i$. Hence we count:
        \begin{itemize}
            \item $|c_i|$ crossings with $\epsilon=\sign(c_i)$ and $d=\sigma_i$;
            \item $|c_i|$ crossings with $\epsilon=\sign(c_i)$ and $d=-\sigma_i$.
        \end{itemize}
        In the $i$-th group of vertical arcs we count one crossing with $\epsilon=\sign(\alpha_i)$ and $d = k$ for each $k\in\{ \sigma_{i-1} + \sign(\alpha_i),\dots, \sigma_i \}$.
        Finally, in the last group of vertical arcs we find one crossing with $\epsilon=-\sign(\sigma_n)$ and $d = k$ for each $k\in\{ \sign(\sigma_n),\dots,\sigma_n \}$.
        Observe that the count of the crossings on the vertical groups of arcs for $i\in\{1,\dots,n\}$ simplify with the count of the crossings on the last group of vertical arcs:
        \[ \sum_{i=1}^n \Big( \sign(\alpha_i) t^{\sigma_{i-1}}\sum_{k=\sign(\alpha_i)}^{\alpha_i/2}  t^{k} \Big) -\sign(\sigma_n) \sum_{k=\sign(\sigma_n)}^{\sigma_n}  t^{k} = 0. \]
        This means that:
        \[ \overline{\eta}(t) = \sum_{i=1}^n c_i\left( t^{\sigma_i} + t^{-\sigma_i} \right) \quad\text{and}\quad \overline{\eta}(1) = 2\sum_{i=1}^n c_i, \]
        hence $\eta(t)= \sum_{i=1}^n c_i \left( t^{\sigma_i} + t^{-\sigma_i} \right) -2\sum_{i=1}^n c_i$.
    \end{proof}
\end{prop}

 In analogy with Sakuma's result \cite[Theorem II]{sakuma} we can observe the following corollary to Proposition~\ref{prop:formula}.

 \begin{cor}\label{surjectivity}
 Every Laurent polynomial in $\Z\langle t\rangle$ is realized as butterfly polynomial of some directed strongly invertible knot.
 \begin{proof}
     Notice that the $\eta$-function of the butterfly link of the $2$-bridge knot $K_n = I_1(2n;1)$ is
     \[ \eta(t) = t^n + t^{-n} -2, \]
     and these form a set of generators for $\mathbb{Z}\langle t\rangle$.
 \end{proof}
 \end{cor}

\begin{remark}
Notice that using only the formula of Proposition \ref{prop:formula} we are not able to deduce that no $2$-bridge knot is equivariantly slice. As an example, observe that the buttefly polynomial vanishes on the family of directed strongly invertible knots given by $I_1(2a,-2a,2a,-2a;b,c,-b,d)$, with $a,b,c,d\in\Z\setminus\{0\}$.
\end{remark}
\section{No two-bridge knot is equivariantly slice}\label{sec:2bri}
In this section we give two different proofs of the following proposition.
\begin{prop}\label{no_2_slice}
The strongly invertible knot $K=I_1(\alpha_1,\dots,\alpha_n;c_1,\dots,c_n)$ is not equivariantly slice.
\end{prop}

The first proof relies on the \emph{axis-linking number} introduced by Boyle and Issa \cite{boyle2021equivariant}, while in the second one we use the \emph{nullity} of the butterfly link as an obstruction to equivariant sliceness.

\subsection{Linking number proof}
\begin{defn}\cite[Definition 4.6]{boyle2021equivariant}
 Let $(K,h,\rho)$ be a directed strongly invertible knot and let $L_b(K)$ be its butterfly link. The \emph{axis-linking number} $\widetilde{\lk}(K)$ is the linking number between one component of $L_b(K)$ and the oriented fixed axis.
\end{defn}
Recall that $\widetilde{\lk}$ is an obstruction to equivariant sliceness, since it induces a group homomorphism $\widetilde{lk}:\C\longrightarrow\Z$ \cite[Proposition 10]{boyle2021equivariant}.

  \begin{proof}[Proof of Proposition \ref{no_2_slice}]
Let $K_n$ be the strongly invertible knot $K$ endowed with the direction described in Figure \ref{I1}, and let $aK_n$ be its antipode. 
      For each $i=1,\dots,n$ let $\delta_i\in\{0,1\}$ such that $c_i = 2k_i + \delta_i$ for some $k_i\in\mathbb{Z}$. For each $i=1,\dots,n$ set $\varepsilon_i^n = \prod_{j=i}^n (-1)^{\delta_i}$.
      A rapid computation shows that
      \[ \widetilde{\lk}(K_n) = \sum_{i=1}^n(\varepsilon_i^n - 1)\alpha_i \quad\text{and}\quad \widetilde{\lk}(aK_n) = \sum_{i=1}^n \varepsilon_i^n \alpha_i. \]
      Figure~\ref{come_butter} shows the two butterfly links.
    \begin{figure}[ht]
       \begin{subfigure}{\textwidth}
       \centering
       \begin{tikzpicture}
          \draw (0,0) node[above right]{\includegraphics[width=\linewidth]{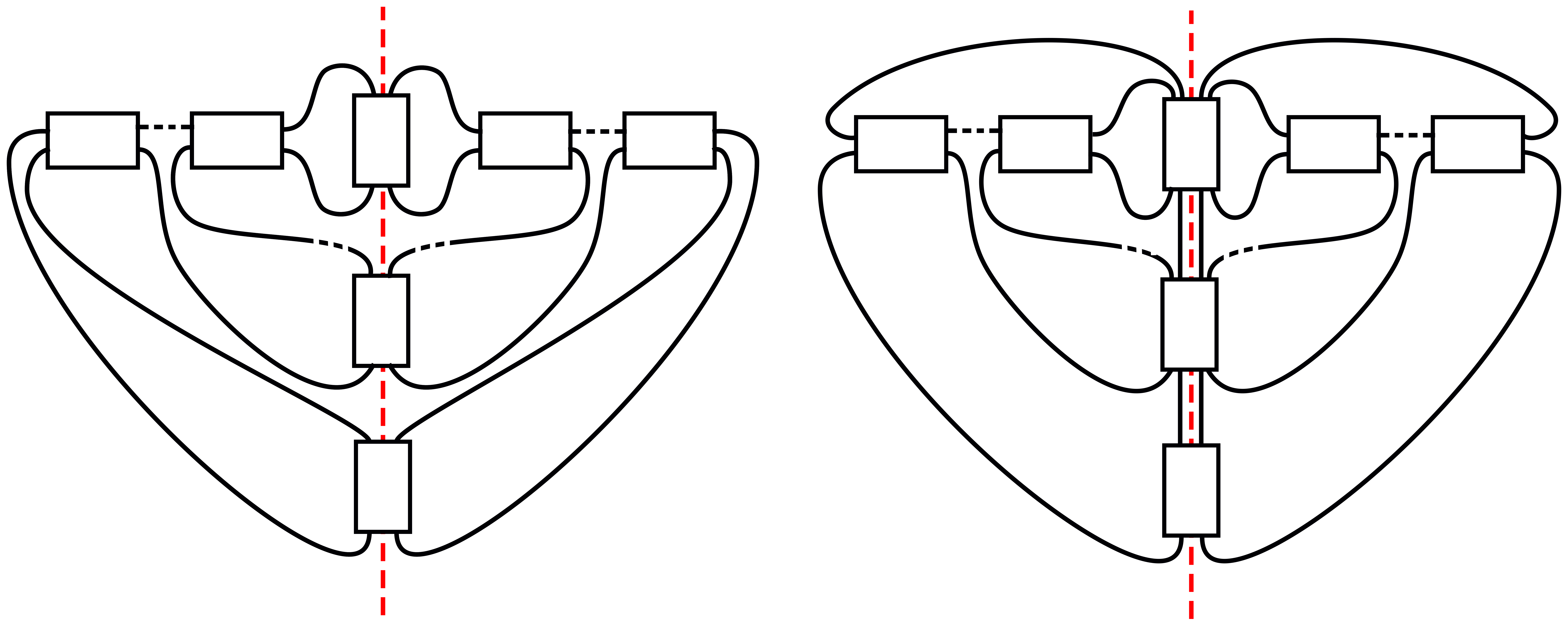}};
          
        \draw (1.01,4.78) node (a) {$-c_n$};
        \draw (2.4,4.78) node (a) {$-c_1$};
        \draw (3.81,4.75) node (a) {$\alpha_1$};
        \draw (5.15,4.78) node (a) {$-c_1$};
        \draw (6.59,4.78) node (a) {$-c_n$};
        \draw (3.79,3) node (a) {$\alpha_n$};
        \draw (3.82,1.47) node (a) {-$b$};
          
        \draw (8.76,4.78) node (a) {$-c_n$};
        \draw (10.15,4.78) node (a) {$-c_1$};
        \draw (11.56,4.75) node (a) {$\alpha_1$};
        \draw (12.9,4.78) node (a) {$-c_1$};
        \draw (14.34,4.78) node (a) {$-c_n$};
        \draw (11.54,3) node (a) {$\alpha_n$};
        \draw (11.55,1.47) node (a) {-$b$};
     \end{tikzpicture}
     \caption{$L_b(K_n)$ and $L_b(K_n')$.}
    \label{come_butter}
    \end{subfigure}
       
       \begin{subfigure}{\textwidth}
       \centering
       \begin{tikzpicture}
          \draw (0,0) node[above right]{\includegraphics[width=0.5\linewidth]{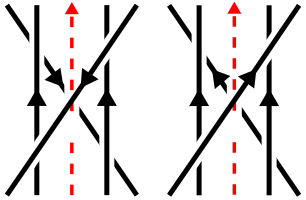}};
          
        \draw (1.85,1.8) node (a) [scale=1.1] {$+1$};
        \draw (0.55,3.8) node (a) [scale=1.1] {$-1$};
        \draw (0.55,1.35) node (a) [scale=1.1] {$-1$};
        \draw (3.15,3.8) node (a) [scale=1.1] {$-1$};
        \draw (3.15,1.35) node (a) [scale=1.1] {$-1$};
        
        \draw (5.9,1.8) node (a) [scale=1.1] {$+1$};
        \draw (4.5,3.8) node (a) [scale=1.1] {$+1$};
        \draw (4.5,1.35) node (a) [scale=1.1] {$+1$};
        \draw (7.2,3.8) node (a) [scale=1.1] {$+1$};
        \draw (7.2,1.35) node (a) [scale=1.1] {$+1$};
     \end{tikzpicture}
     \caption{Crossings in the butterfly link.}
    \label{incroci}
    \end{subfigure}
    \caption{}
    \end{figure}
      In the case where the fixed half-axis is the one going through infinity the computation is obvious. For the other choice of half-axis, Figure~\ref{incroci} shows the crossings we get in correspondence of each crossing inside the $\alpha_i$-boxes. 
      Since being equivariantly slice does not depend on the choice of the direction, it follows that $K_n$ cannot be equivariantly slice if either $\widetilde{\lk}(K_n)$ or $\widetilde{\lk}(aK_n)$ is non-zero.

      If both of these invariants are zero we get that $b=\sum_{i=1}^n \alpha_i =\sum_{i=1}^n\varepsilon_i^n\alpha_i = 0$. Let $K_{n-1}$ be the strongly invertible knot $I_1(\alpha_1,\dots,\alpha_{n-1};c_1,  \dots,c_{n-1})$. 
      Figure~\ref{labello} shows that under this assumption $L_b(K_n)$ and $L_b(K_{n-1})$ are equivariantly isotopic (as $2$-periodic $2$-component links).
      
     \begin{figure}[ht]
        \centering
         \begin{tikzpicture}
          \draw (0,0) node[above right]{\includegraphics[width=0.85\linewidth]{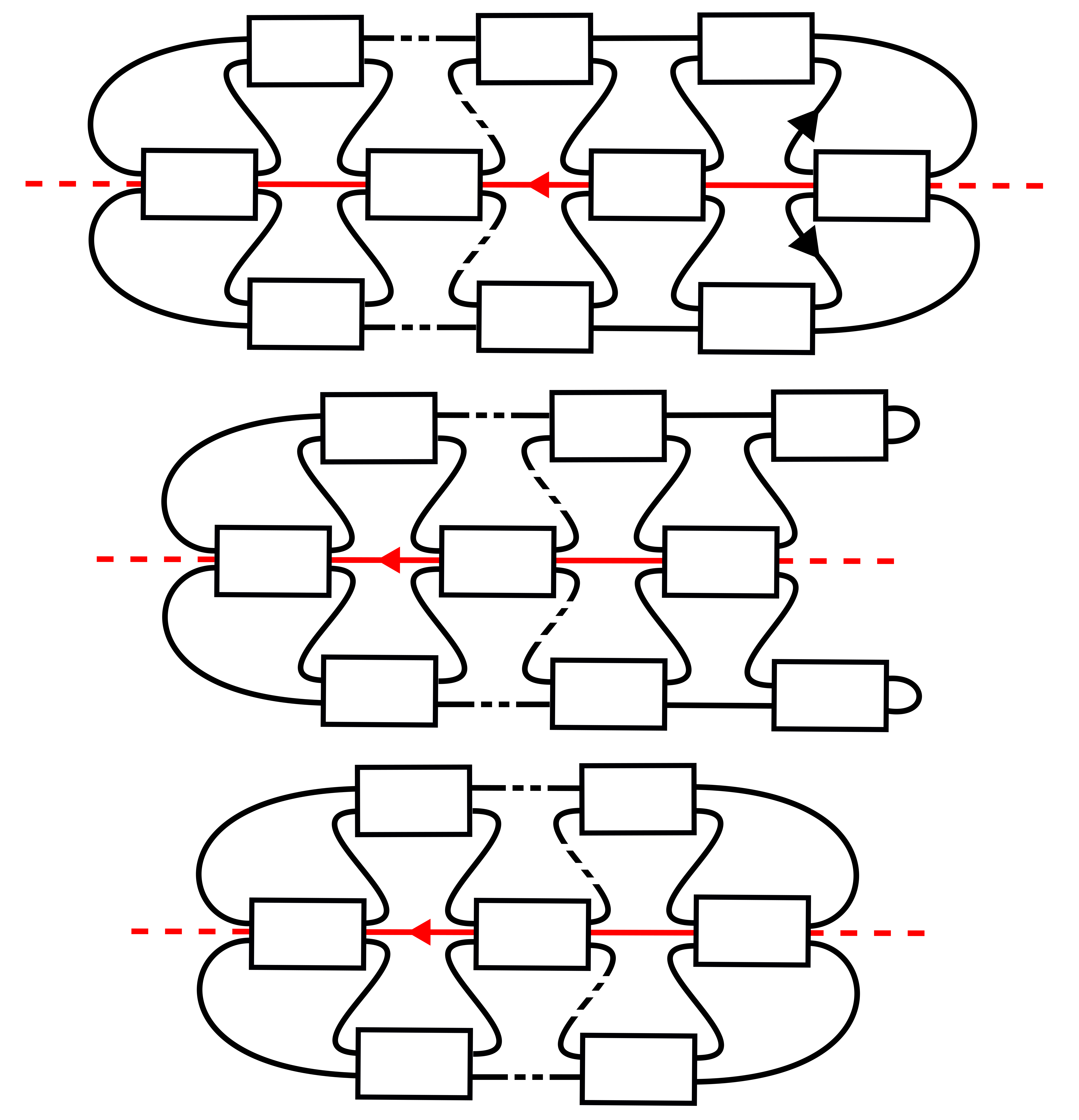}};
   		\draw (3.8,12.77) node (a) [scale=1.3] {-$c_1$};
        \draw (6.5,12.77) node (a) [scale=1.3] {-$c_{n-1}$};
        \draw (9.17,12.77) node (a) [scale=1.3] {-$c_n$};
   		\draw (3.8,9.67) node (a) [scale=1.3] {-$c_1$};
        \draw (6.5,9.67) node (a) [scale=1.3] {-$c_{n-1}$};
        \draw (9.17,9.65) node (a) [scale=1.3] {-$c_n$};
        
   		\draw (4.6,5.15) node (a) [scale=1.3] {-$c_1$};
        \draw (7.35,5.15) node (a) [scale=1.3] {-$c_{n-1}$};
        \draw (9.97,5.15) node (a) [scale=1.3] {-$c_n$};
   		\draw (4.6,8.3) node (a) [scale=1.3] {-$c_1$};
        \draw (7.35,8.3) node (a) [scale=1.3] {-$c_{n-1}$};
        \draw (9.97,8.3) node (a) [scale=1.3] {-$c_n$};
        
   		\draw (5,3.85) node (a) [scale=1.3] {-$c_1$};
        \draw (7.75,3.85) node (a) [scale=1.3] {-$c_{n-1}$};
        \draw (5,0.75) node (a) [scale=1.3] {-$c_1$};
        \draw (7.75,0.7) node (a) [scale=1.3] {-$c_{n-1}$};
        
   		\draw (2.465,11.2) node (a) [scale=1.3] {$\alpha_1$};
        \draw (5.15,11.2) node (a) [scale=1.3] {$\alpha_2$};
        \draw (7.835,11.2) node (a) [scale=1.3] {$\alpha_{n}$};
        \draw (10.43,11.2) node (a) [scale=1.3] {$0$};
        
        \draw (3.3,6.73) node (a) [scale=1.3] {$\alpha_1$};
        \draw (5.97,6.73) node (a) [scale=1.3] {$\alpha_2$};
        \draw (8.64,6.73) node (a) [scale=1.3] {$\alpha_{n}$};
        
        \draw (3.8,2.3) node (a) [scale=1.3] {$\alpha_1$};
        \draw (6.47,2.3) node (a) [scale=1.3] {$\alpha_2$};
        \draw (9.11,2.3) node (a) [scale=1.3] {$\alpha_{n}$};
           \end{tikzpicture}
           \caption{Equivariant isotopy between $L_b(K_n) $ and $ L_b(K_{n-1})$.}
        \label{labello}
     \end{figure}
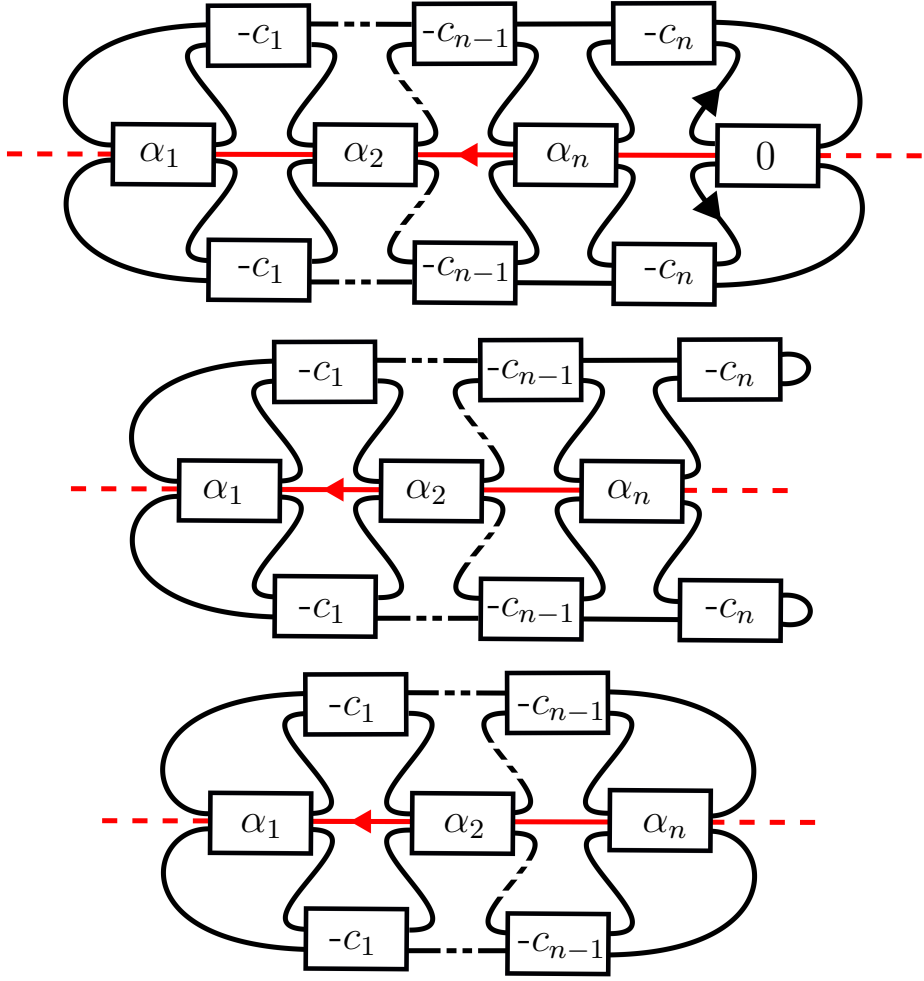
      Hence $\widetilde{\lk}(K_{n-1}) = 0$, but
      \[ \widetilde{\lk}(aK_{n-1}) = \sum_{i=1}^{n-1} \varepsilon_i^{n-1}\alpha_i = \pm \sum_{i=1}^{n-1} \varepsilon_i^n \alpha_i = \mp \varepsilon_n^n\alpha_n \neq 0, \]
      since $\alpha_n\neq0$ by the assumption on the continued fraction of $K_n$.
      This implies that $K_{n-1}$ and $aK_{n-1}$ are not equivariantly slice and hence by Corollary \ref{iff_equiv_slice} that $L_b(K_{n-1})=L_b(K_n)$ is not strongly equivariantly slice. By Proposition \ref{b_link_concordance} or \cite[Proposition 7]{boyle2021equivariant} it follows that $K_n$ is not equivariantly slice. 
  \end{proof}

\begin{remark}
Even though the proof above of Proposition \ref{no_2_slice} uses a homomorphism to prove that a $2$-bridge knot $K$ is not equivariantly slice, it is not clear if it can be adapted to show that $K$ has infinite order in $\C$.
\end{remark}

\subsection{Nullity proof}\label{nullity_proof}
Let $L\subset S^3$ be a link, and denote by $\Sigma(L)$ the $2$-fold cover of $S^3$ branched over $L$.
Recall that the \emph{nullity} of $L$ is defined as
$$n(L)=1+\dim(H_1(\Sigma(L),\Q)),$$
and that the nullity is an invariant for link concordance, as shown in \cite{kauffman1976signature}.

\begin{proof}[Proof of Proposition \ref{no_2_slice}]
Consider on $K$ the direction specified in Figure \ref{I1}.
Recall that the fraction associated with $K$ is
$$p/q=[\alpha_1,-2c_1,\dots,\alpha_n,-2c_n].$$
Observe from Figure \ref{label} that $L_b(K)$ is a $2$-bridge link with continued fraction $[\alpha_1,-2c_1,\dots,\alpha_n,-2c_n,-\sum_{i=1}^n\alpha_i]$.

Notice that the continued fraction $[-\sum_{i=1}^n\alpha_i,-2c_n,\alpha_n,\dots,-2c_1,\alpha_1]$ gives the same $2$-bridge link and that the associated rational number is
$$p''/q''=-\sum_{i=1}^n\alpha_i+q'/p',$$
where $p'/q'=[-2c_n,\alpha_n,\dots,-2c_1,\alpha_1]$.
It is a well known fact (see \cite[Theorem 4]{kauffman2002classification} or \cite[Exercise 2.1.12]{kawauchisurvey}) that $p=p'$ and $q'$ is such that $q\cdot q'\equiv -1 \mod{p}$.
It follows that $p''/q''\in\Q\setminus\{0\}$. Since the $2$-fold cover $\Sigma(L_b(K))$ of $S^3$ branched over $L_b(K)$ is a lens space and $p''/q''\neq0$, we have that $\Sigma(L_b(K))$ is a rational homology $S^3$.
In particular the nullity of $L_b(K)$ is $n(L_b(K))=1$.
Since the nullity of the $2$-component unlink is easily seen to be $2$, it follows that $L_b(K)$ is not concordant to the unlink.
Therefore $K$ is not equivariantly slice.
\end{proof}

\section{A new invariant of equivariant concordance}\label{sec:moth}
Recall that a \emph{semi-orientation} on a link $L$ is the choice of an orientation on each component of $L$, up to reversing the orientation on all components simultaneously.

\begin{defn}
Let $K$ be a directed strongly invertible knot. We define $\widehat{L_b}(K)$ to be the $2$-periodic, semi-oriented link obtained by endowing $L_b(K)$ with the opposite semi-orientation.
\end{defn}

Observe that $\widehat{L_b}(K)$ is obtained from $K$ via a band move coherent with the (unique) semi-orientation of $K$.
Viceversa, we can attach another equivariant band $B$ to $\widehat{L_b}(K)$ obtaining again $K$ (see Figure~\ref{moth}).

\begin{figure}[ht]
    \centering
    \begin{tikzpicture}
         \draw (0,0) node[above right]{\includegraphics[width=\linewidth]{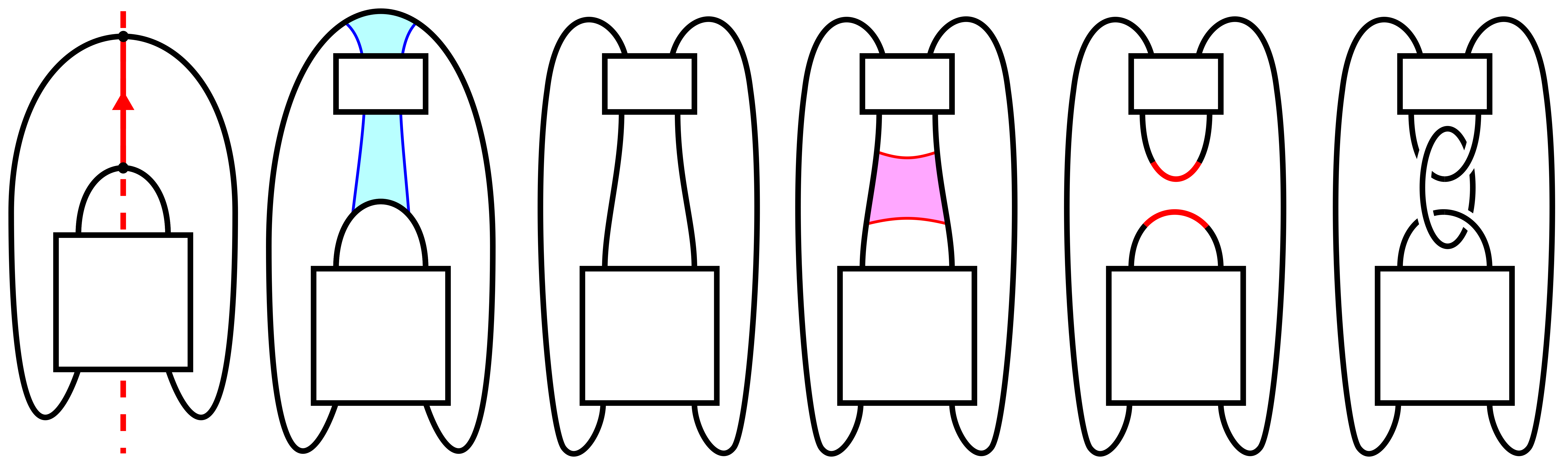}};
   		\draw (1.3,1.65) node (a) [scale=1.3] {$K$};
        \draw (3.75,1.3) node (a) [scale=1.3] {$K$};
        \draw (6.35,1.3) node (a) [scale=1.3] {$K$};
        \draw (8.8,1.3) node (a) [scale=1.3] {$K$};
        \draw (11.35,1.3) node (a) [scale=1.3] {$K$};
        \draw (13.95,1.3) node (a) [scale=1.3] {$K$};

        \draw (13.47,2.7) node (a) [scale=1.3] {$U$};
        \draw (8.8,2.72) node (a) [scale=1.2] {$B$};
        
        \draw (3.75,3.72) node (a) [scale=1] {$-b$};
        \draw (6.35,3.72) node (a) [scale=1] {$-b$};
        \draw (8.8,3.72) node (a) [scale=1] {$-b$};
        \draw (11.35,3.72) node (a) [scale=1] {$-b$};
        \draw (13.95,3.72) node (a) [scale=1] {$-b$};

        \draw (6.35,-0.27) node (a) [scale=1] {$\widehat{L_b}(K)$};
        \draw (13.95,-0.3) node (a) [scale=1] {$L_m(K)$};
        
    \end{tikzpicture}
    \caption{Construction of the moth link.}
    \label{moth}
\end{figure}
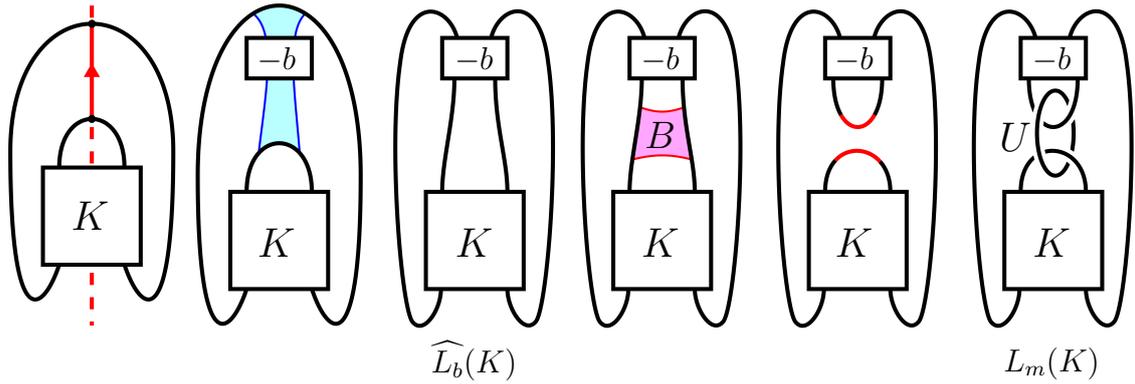

\begin{defn}
We define the \emph{moth link of} $K$ to be the link $L_m(K)$ given by the union of $K$ and a meridian $U$ of the core of the band $B$, as described in Figure~\ref{moth}.
Observe that such meridian can be chosen so that $L_m(K)$ is a $2$-component strongly invertible link.
\end{defn}

\begin{remark}\emph{
Using the notation of \cite{kaiser1992strong}, we have that $L_m(K)$ is the \emph{strong fusion} of the link $\widehat{L_b}(K)$ along the band $B$.}
\end{remark}

\begin{prop}\label{prop:conc_inv}
Let $(K_0,\rho_0,h_0)$ and $(K_1,\rho_1,h_1)$ be two equivariantly concordant directed strongly invertible knots. Then $L_m(K)$ is equivariantly concordant to $L_m(J)$.
\end{prop}

\begin{proof}
Let $C\subset S^3\times I$ be a concordance between $(K_0,\rho_0,h_0)$ and $(K_1,\rho_1,h_1)$ equivariant with respect to an extension $\rho:S^3\times I\longrightarrow S^3\times I$ of $\rho_0\sqcup\rho_1$.
In the proof of Proposition \ref{b_link_concordance} we found an equivariant embedding of $E\cong D^1\times D^1\times D^1$ in $S^3\times I$ which intersects the concordance $C$ in $D^1\times\partial D^1\times D^1$ and such that $C_b=(C\setminus E)\cup\partial D^1\times D^1\times D^1$ exhibits an equivariant concordance between $L_b(K_0)$ and $L_b(K_1)$.

Now let $N$ be an tubular neighbourhood of $D^1\times 0\times D^1$ in $S^3\times I$ and let $c$ be the arc $0\times 0\times D^1$.
We can take $\epsilon>0$ small such that $E_\epsilon=D^1\times \epsilon D^1\times D^1\subset N$.
Then $(C_b\setminus E_\epsilon)\cup D^1\times \partial(\epsilon D^1)\times D^1\cup \partial N_{|c}$ is an equivariant concordance between $L_m(K_0)$ and $L_m(K_1)$.
\end{proof}

\begin{defn}
Let $K$ be a directed strongly invertible knot. We define the \emph{moth polynomial} of $K$ as the $\eta$-function of $L_m(K)=K\cup U$, taken with respect to the component $K$, i.e.
$$\eta(L_m(K))(t)=\eta(K,U;t).$$
\end{defn}

\begin{prop}
The moth polynomial induces a group homomorphism
$$\eta(L_m(-)):\C\longrightarrow\Q(t).$$
\end{prop}

\begin{proof}
Let $K$ and $J$ be two directed strongly invertible knots. By Proposition~\ref{prop:conc_inv} if $K$ and $J$ are equivariantly concordant then $L_m(K)$ and $L_m(J)$ are concordant. Since the $\eta$-function is a concordance invariant we have that $\eta(L_m(K))=\eta(L_m(J))$, therefore $\eta(L_m(-))$ is well defined.
Next we have to show that $\eta(L_m(K\widetilde{\#}J))=\eta(L_m(K))+\eta(L_m(J))$. This follows by observing that $L_m(K\widetilde{\#}J)$ is obtained from $L_m(K)$ and $L_m(J)$ by a band sum, as shown in Figure~\ref{band_sum} and using \cite[Theorem 7.1]{cochran1985geometric}.
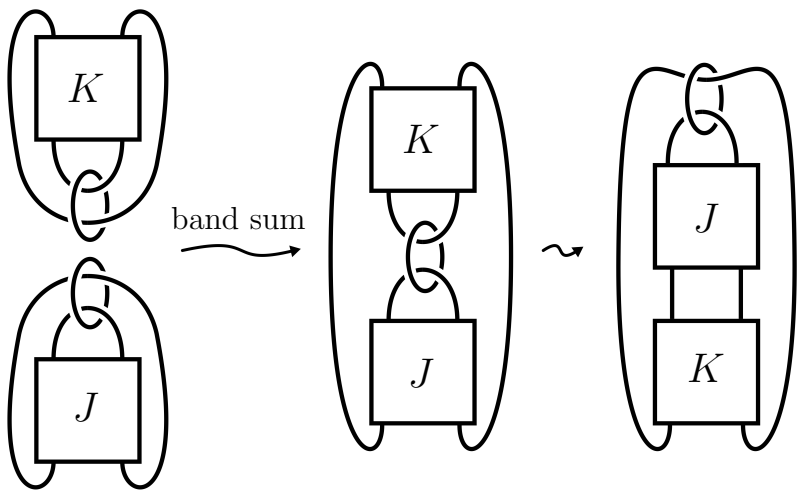
\begin{figure}[ht]
    \centering
    \begin{tikzpicture}
         \draw (0,0) node[above right]{\includegraphics[width=0.7\linewidth]{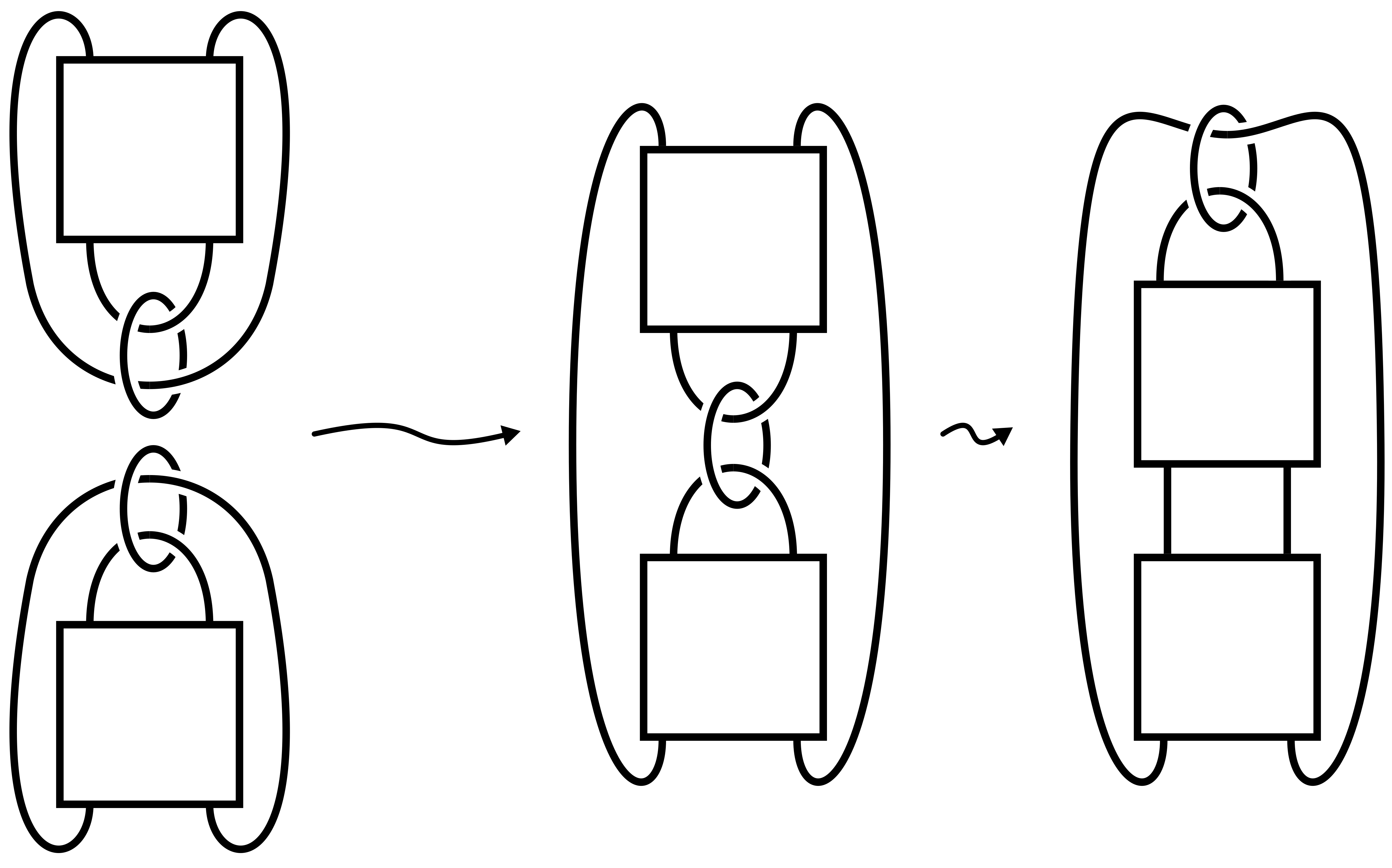}};

   		\draw (1.24,5.58) node (a) [scale=1.3] {$K$};
   		\draw (1.24,1.35) node (a) [scale=1.3] {$J$};

   		\draw (3.25,3.85) node (a) [scale=1] {band sum};
     
   		\draw (5.65,4.9) node (a) [scale=1.3] {$K$};
   		\draw (5.65,1.8) node (a) [scale=1.3] {$J$};
     
   		\draw (9.4,1.86) node (a) [scale=1.3] {$K$};
   		\draw (9.4,3.85) node (a) [scale=1.3] {$J$};
    \end{tikzpicture}
    \caption{The band sum of $L_m(K)$ and $L_m(J)$ is $L_m(K\widetilde{\#}J)$.}
    \label{band_sum}
\end{figure}
\end{proof}

The following proposition is an immediate consequence of \cite[Proposition 1]{kaiser1992strong} and \cite[Theorem 7.1]{cochran1985geometric}.
\begin{prop}
The moth polynomial of a directed strongly invertible knot $K$ can be computed by the following formula:
$$\eta(L_m(K))(t)=\frac{\nabla_{\widehat{L_b}(K)}(z)}{z\nabla_K(z)},$$
where $\nabla_L(z)$ is the Conway polynomial of an oriented (or semi-oriented) link $L$ and $z=i(t^{1/2}-t^{-1/2})$.
\end{prop}

Since $\eta(L_m(-))$ is a homomorphism, the proposition above implies the following result.

\begin{thm}\label{infinite_order}
Let $K$ be a directed strongly invertible knot such that $\nabla_{\widehat{L_b}(K)}(z)\neq0$. Then $K$ is not equivariantly slice and has infinite order in $\C$.
\end{thm}

As an immediate application of Theorem \ref{infinite_order} we have the following refinement of the results in Section~\ref{sec:2bri} on $2$-bridge knots.

\begin{prop}\label{prop:2bri}
Every $2$-bridge knot has infinite order in $\C$, independently of the choice of strong inversion and direction.
\end{prop}
\begin{proof}
First of all, observe that by Remark \ref{rem_sliceness} it is sufficient to show that a directed strongly invertible knot $K$ of type $I_1(\alpha_1,\dots,\alpha_n,c_1,\dots,c_n)$, with the direction specified in Figure~\ref{I1a}, has infinite order in $\C$.

As proven in Subsection \ref{nullity_proof} we know that
$\Sigma(L_b(K))=\Sigma(\widehat{L_b}(K))$ is a rational homology $S^3$.

Recall now that for a link $L\subset S^3$ we have that $|H_1(\Sigma(L),\Z)|=|\Delta_L(-1)|$, where $0$ means that the group is infinite.
Since $H_1(\Sigma(\widehat{L_b}(K)),\Z)$ is finite, we deduce that the Alexander polynomial of $\widehat{L_b}(K)$ is non-zero, and hence by Theorem \ref{infinite_order} that $K$ has infinite order in $\C$.
\end{proof}

\bibliographystyle{alpha} 
\bibliography{refs} 
\end{document}